\documentclass[1pt]{article}
\usepackage{amsmath, amssymb, amsthm, amsfonts, cases}
\usepackage{mathrsfs}
\usepackage{url}
\usepackage{authblk}
\usepackage[usenames]{color}
\usepackage{geometry}
\theoremstyle{plain}
\newtheorem{thm}{Theorem}[section]

\newtheorem{pro}[thm]{Problem}
\newtheorem{lem}[thm]{Lemma}

\theoremstyle{definition}

\newtheorem{ass}{Assumption}[section]
\newtheorem{rmk}{Remark}[section]

\newcommand{\eps}{\varepsilon}

\makeatletter\@addtoreset{equation}{section} \makeatother

\begin{document}

\title{Linear-Quadratic Optimal Control Problems for Mean-Field
Stochastic Differential Equations with Jumps
\thanks{This work was supported by the Natural Science Foundation of Zhejiang Province
for Distinguished Young Scholar  (No.LR15A010001),  andthe National Natural
Science Foundation of China (No.11471079, 11301177) }}

\date{}

\author[a]{Maoning Tang}
\author[a]{Qingxin Meng\footnote{Corresponding author.
\authorcr
\indent E-mail address: mqx@hutc.zj.cn(Q. Meng)}}

\affil[a]{\small{Department of Mathematical Sciences, Huzhou University, Zhejiang 313000, China}}

\maketitle

\begin{abstract}
\noindent
In this paper, we study a linear-
quadratic optimal control problem 
for mean-field stochastic 
differential equations 
driven by a Poisson random martingale measure and a multidimensional
Brownian motion. Firstly, the existence and 
uniqueness of the optimal control is 
obtained by the classic  convex variation 
principle. Secondly, by the duality method, the optimality system, also called the stochastic Hamilton system which turns out to be
a linear fully coupled mean-field forward-backward stochastic differential equation with jumps, is derived to characterize the optimal control.
Thirdly, applying a decoupling technique, 
we establish the connection between two Riccati equation and the stochastic Hamilton system and  then prove the optimal control  has a state feedback representation.

\end{abstract}

\section{Introduction}

\subsection{Notations}
~~~~Let $T$ be a fixed strictly positive real number and  $(\Omega,
\mathscr{F},\{\mathscr{F}_t\}_{0\leq t\leq T}, P)$ be a complete probability space on which a d-dimensional standard Brownian motion $\{W(t), 0\leq t\leq T\}$
is defined. Denote by $\mathscr{P}$
the $\mathscr{F}_t$-predictable $\sigma$-field on $[0, T]\times \Omega$ and by $\mathscr B(\Lambda)$
  the Borel $\sigma$-algebra of any topological space $\Lambda.$ Let $(Z,\mathscr B (Z), v
)$ be a measurable space with $v(Z)<\infty $ and $\eta: \Omega\times D_\eta \longrightarrow Z$ be an
$\mathscr F_t$-adapted stationary Poisson point process with characteristic measure $v$, where
 $D_\eta$  is a countable subset of $(0, \infty)$. Then the counting measure induced by $\eta$ is
$$
\mu((0,t]\times A):=\#\{s\in D_\eta; s\leq t, \eta(s)\in A\},~~~for~~~ t>0, A\in \mathscr B (Z).
$$
And $\tilde{\mu}(de,dt):=\mu(de,dt)-v(de)dt$ is a compensated Poisson random martingale  measure which
is assumed to be independent of Brownian motion
$\{W(t), 0\leq t\leq T\}$. Assume $\{\mathscr{F}_t\}_{0\leq t\leq T}$ is $P$-completed natural
 filtration generated by $\{W(t),
0\leq t\leq T\}$ and $\{\iint_{A\times (0,t] }\tilde{\mu}(de,ds), 0\leq t\leq T, A\in \mathscr B (Z) \}.$  In the following,  we introduce the basic
notation used throughout this paper.

$\bullet$~~$H$: a Hilbert space with norm $\|\cdot\|_H$.

$\bullet$~~$\langle\alpha,\beta\rangle:$ the inner product in
$\mathbb{R}^n, \forall \alpha,\beta\in\mathbb{R}^n.$

$\bullet$~~$|\alpha|=\sqrt{\langle\alpha,\alpha\rangle}:$ the norm
of $\mathbb{R}^n,\forall \alpha\in\mathbb{R}^n.$

$\bullet$~~$\langle A,B\rangle=tr(AB^\top):$ the inner product in
$\mathbb{R}^{n\times m},\forall A,B\in \mathbb{R}^{n\times m}.$
Here denote by $B^\top$ the transpose of a matrix B.

$\bullet$~~$|A|=\sqrt{tr(AA^\top)}:$ the norm of $A$.

$\bullet$~~$S^n:$ the set of all $n\times n$ symmetric matrices.

$\bullet$~~ $S^n_+:$ the subset of all non-negative definite matrices of $S^n.$

$\bullet$~~$S_{\mathscr{F}}^2(0,T;H):$ the space of all $H$-valued
and ${\mathscr{F}}_t$-adapted  c\`{a}dl\`{a}g processes
$f=\{f(t,\omega),\ (t,\omega)\in[0,T]\times\Omega\}$ satisfying
$$\|f\|_{S_{\mathscr{F}}^2(0,T;H)}^2\triangleq{
\mathbb E\displaystyle\sup_{0 \leq t \leq T}\|f(t)\|_H^2dt}<+\infty.$$

$\bullet$~~$ M_{\mathscr{F}}^2(0,T;H):$ the space of all $H$-valued
and ${\mathscr{F}}_t$-adapted processes $f=\{f(t,\omega),\
(t,\omega)\in[0,T]\times\Omega\}$ satisfying
$$\|f\|_{M_{\mathscr{F}}^2(0,T;H)}^2\triangleq {\mathbb E\displaystyle\bigg[\int_0^T\|f(t)\|_H^2dt}\bigg]<\infty.$$

$\bullet$~~${M}^{\nu,2}( Z; H):$ the space of all H-valued measurable
  functions $r=\{r(\theta), \theta \in Z\}$ defined on the measure space $(Z, \mathscr B(Z); v)$ satisfying
$$\|r\|_{{ M}^{\nu,2}( Z; H)}^2\triangleq{\displaystyle\int_Z\|r(\theta)\|_H^2v(d\theta)}<~\infty.$$

 $\bullet$~~ ${M}_{\mathscr{F}}^{\nu,2}{([0,T]\times  Z; H)}:$ the  space of all ${M}^{\nu,2}( Z; H)$-valued
and ${\cal F}_t$-predictable processes $r=\{r(t,\omega,e),\
(t,\omega,e)\in[0,T]\times\Omega\times Z\}$ satisfying
$$\|r\|_{M_{\mathscr F}^{\nu,2}([0,T]\times  Z; H)}^2\triangleq {\mathbb E\bigg[\int_0^T\displaystyle\|r(t,\cdot)\|^2_
{{M}^{\nu,2}( Z; H)}dt}\bigg]<~\infty.$$
$\bullet$~~$L^2(\Omega,{\mathscr {F}},P;H):$ the space of all
$H$-valued random variables $\xi$ on $(\Omega,{\mathscr  {F}},P)$
satisfying
$$\|\xi\|_{L^2(\Omega,{\mathscr{F}},P;H)}\triangleq \mathbb E[\|\xi\|_H^2]<\infty.$$

\subsection{ Formulation of Problem}

Consider the following linear stochastic system driven by
 Brownian motion $\{W(t)\}_{0\leq t\leq T}$ and Poisson random  martingale  measure $\{\tilde{\mu}(d\theta,dt)\}_{0\leq t\leq T}$
\begin{equation}\label{eq:1.1}
\left\{\begin {array}{ll}
  dX(t)=&(A(t)X(t)+\bar A(t)\mathbb E [X(t)]
  +B(t)u(t)+\bar B(t)\mathbb E [u(t)])dt
  \\&+(C(t)X(t)
  +\bar C(t)\mathbb E [X(t)]
  +D(t)u(t)+\bar D(t)\mathbb E [u(t)])dW(t)
  \\&+\displaystyle
  \int_{Z}(E(t, \theta)X(t-)+\bar E(t, \theta)\mathbb E [X(t-)]
  +F(t,\theta)u(t)+\bar F(t, \theta)\mathbb E[u(t)])\tilde{\mu}(d\theta,
  dt),
   \\x(0)=&x \in \mathbb R^n,
\end {array}
\right.
\end{equation}
with the following quadratic cost {functional}
\begin{eqnarray}\label{eq:1.2}
\begin{split}
 J(x, u(\cdot))=&\displaystyle \mathbb E\bigg[\int_0^T\bigg(\langle Q(t)X(t),
X(t)\rangle+ \langle \bar {
Q}(t)\mathbb E[X(t)], \mathbb E[X(t)]\rangle +\langle N(t)u(t), u(t)\rangle\\&+\langle \bar
{N}(t)\mathbb E[u(t)], \mathbb E[u(t)]\rangle \bigg)dt\bigg]+\mathbb E[\langle
GX(T), X(T)\rangle] +\langle \bar{ G}\mathbb E[X(T)], \mathbb
E[X(T)]\rangle,
\end{split}
\end{eqnarray}
where $A(\cdot), \bar A(\cdot), B(\cdot), \bar B(\cdot),C(\cdot), \bar C(\cdot), D(\cdot), \bar D(\cdot)
, E(\cdot,\cdot), \bar E(\cdot,\cdot),
F(\cdot,\cdot),\bar F(\cdot), Q(\cdot), \bar Q(\cdot),
 N(\cdot), \bar N(\cdot) $  are given
 matrix valued deterministic functions, and
 $G$ and $\bar G$ are given matrices.

In the above, $u(\cdot)$ is our admissible control process. In this paper, a predictable stochastic process  $u(\cdot)$ is said to be an admissible control, if $ u(\cdot)\in
M_{\mathscr F}^2(0, T;\mathbb R^m)$. The set of all admissible controls is denoted by ${\cal A}$ .
 For any admissible control $u(\cdot),$ the strong solution of the system { \eqref{eq:1.1}},  denoted by $X^{(x,u)}(\cdot)$ or $X(\cdot)$
  if its dependence on
  admissible control $u(\cdot)$ is clear from  the context,   is called the
state process corresponding to the control process $u(\cdot)$, and
 $(u(\cdot), X(\cdot))$ is called an
admissible pair.

Our optimal  control problem can be stated
as follows:
\begin{pro}\label{pro:1.1}
For given $x\in \mathbb R^n,$
find an admissible control ${u}^*(\cdot)$ such that
\begin{equation}\label{eq:b7}
J(x,{u}^*(\cdot))=\displaystyle\inf_{u(\cdot)\in {\cal A}}J(x, u(\cdot)).
\end{equation}
\end{pro}

Any  ${u}^*(\cdot)\in {\cal A}$ satisfying  the above is called an
optimal control process of Problem \ref{pro:1.1} and the
corresponding state process $X^*(\cdot)$ is
called the optimal state process. Correspondingly $(u^*(\cdot),
X^*(\cdot))$ is called an optimal pair of
Problem
\ref{pro:1.1}.

Note that  $\mathbb E [X(\cdot)]$ and $ \mathbb E [u(\cdot)]$  appear in  the state equation
and cost functional.  Such a state equation is
referred to as  a mean-field stochastic differential equation (MF-SDE).
 For details on motivations for the inclusion of $\mathbb E [X(\cdot)]$ and $ \mathbb E[u(\cdot)]$ in the
 cost functional, the the interested
  reader is referred to [28].

Throughout this paper, we make the following assumptions on the coefficients

\begin{ass}\label{ass:1.1}
 The matrix-valued functions $A, \bar A, C, \bar C,Q, \bar Q:[0, T]\rightarrow \mathbb R^{n\times n};
B,\bar B,D, \bar D :[0, T]\rightarrow \mathbb R^{n\times m}; E, \bar E:[0, T]\rightarrow {\cal L}^{v,2}(Z; \mathbb R^{n\times n}), F, \bar F:[0,T]\rightarrow  {\cal L} ^{v,2}(Z; \mathbb R^{n\times m}); N, \bar N:[0, T]\rightarrow \mathbb R^{m\times m}$ are uniformly bounded measurable functions.
\end{ass}

\begin{ass}\label{ass:1.2}
 The matrix-valued functions $Q, Q+\bar Q, N, N+\bar N$ are a.e. nonnegative
matrix, and $M, M+\bar M$ are nonnegative matrices.  Moreover, $N,N+\bar N $  uniformly positive, i.e. for $\forall u\in \mathbb R^m$ and a.s. $t\in [0, T]$,
$ \langle N(t)u, u \rangle \geq \delta \langle u, u\rangle
$ and $ \langle (N(t)+\bar N(t))u, u \rangle \geq \delta \langle  u, u\rangle,
$  for some positive constant
$\delta.$
\end{ass}

The following result gives the well-posedness of the state equation as well as some useful estimates.

\begin{lem}\label{lem:1.1}
Let Assumption \ref{ass:1.1}  be satisfied. Then for any admissible control $u(\cdot)$,
the state equation
\eqref{eq:1.1}
has a unique solution $X(\cdot) \in S_{\mathscr F}^2 ( 0, T; \mathbb R^n).$
Moreover, we have the following estimate
\begin{eqnarray}\label{eq:1.4}
\begin{split}
  \mathbb E \Big[\sup_{0\leq t\leq T}| X(t)|^2\Big] \leq K \mathbb E\bigg[\int_0^T
  |u(t)|^2 dt+x\bigg]
  \end{split}
\end{eqnarray}
and
\begin{eqnarray}\label{eq:1.5}
\begin{split}
  |J(x, u(\cdot))|< \infty.
  \end{split}
\end{eqnarray}
Suppose that $ \bar X(\cdot)$  be the state process corresponding to
another admissible control $\bar u(\cdot),$  then we have the
following estimate
\begin{eqnarray} \label{eq:1.6}
\begin{split}
  \mathbb E \bigg[\sup_{0\leq t\leq T}| X(t)- \bar X(t)|^2\bigg]\leq K \mathbb E\bigg[\int_0^T
  |u(t)-\bar u(t)|^2dt\bigg].
  \end{split}
\end{eqnarray}
\end{lem}

\begin{proof}
  The existence and uniqueness of the  solution can be  obtained by a standard argument using the contraction mapping theorem.
  For the estimates \eqref{eq:1.4} and \eqref{eq:1.6}, we can
  easily  obtain them  by  applying the It\^{o} formula to $|X(\cdot)|^2$ and
   $|X(\cdot)-\bar X(\cdot)|^2 $, Gronwall inequality and B-D-G inequality.
   For the  estimate \eqref{eq:1.5}, using  Assumption
   \ref{ass:1.1} and  the  estimate  \eqref{eq:1.4},   we have

   \begin{eqnarray}\label{eq:1.7}
     \begin{split}
       |J(x, u(\cdot))|\leq K \bigg\{\mathbb E \bigg[\sup_{0\leq t\leq T}| X(t)|^2 \bigg] +\mathbb E\bigg[\int_0^T
  |u(t)|^2 dt\bigg]\bigg\} \leq K \mathbb E\bigg[\int_0^T
  |u(t)|^2 dt+x\bigg]<\infty.
     \end{split}
   \end{eqnarray}
   where we have used the elementary inequality:
   for any $\xi \in {L^2(\Omega,{\mathscr {F}},P;H)}$
   \begin{eqnarray}
     \begin{split}
       ||\mathbb E \xi||^2_H\leq  \mathbb E||\xi||_H^2
     \end{split}
   \end{eqnarray}
The proof is complete.
\end{proof}
Therefore, by Lemma \ref{lem:1.1}, we know that
Problem \ref{pro:1.1}
is well-defined.

\subsection{ Related Development and Contributions of
this paper }
 Most recently, thanks to   comprehensive  practical applications such as in
  economics and finance,  optimal control problems of mean-field type  become a
  popular topic and are studied
  by many researcher. The main feature of this
  type of problems is that the coefficients of
   the state equation  and cost  functional depend not only on the state and the control but also on their  probability distribution 
 As described in [4], due to the mean-field term involved in
 the cost functional, the corresponding optimal control problem become to be  a time-inconsistent
  optimal control problem where the dynamic programming principle (DPP) is not  effective  which  makes many  researchers to solve this type of optimal control problems by establishing
  the stochastic maximum  principle (SMP)
  instead of trying extensions of DPP.
    Interested readers may refer to 
    [1-6],[9],[10], [12],[16], [17],
    [22-25],[27] for various versions of the stochastic maximum
principles for the mean-field models.

In 2013, Yong [28] systematically studied the continuous-time mean-field LQ control problem,
where the optimal control is represented as a state feedback form by introducing two Riccati differential equations. Since [28], abundant results and advances have been made on mean-field LQ control problem (cf., for example, [8], [11], [18-21],[26]).
Different the above mentioned references, the purpose of this paper is to
 extend  continuous-time mean-field LQ control 
 problem to jump diffusion system \eqref{eq:1.1} and establish the corresponding theoretical results.
We first establish the existence and uniqueness of the optimal control by classic 
convex variation principle in Section 2. Then, in Section 3, we will establish
the dual characterization of the optimal control by optimality system, also called stochastic Hamiltonian system. Here  the stochastic Hamiltonian system turns out to be a
coupled  forward-backward stochastic differential equation of mean-field type 
with jump,
consisting  of the state equation, the adjoint equation and the dual
presentation of the optimal control.
Although  the  stochastic Hamilton system gives a
complete characterization of the stochastic LQ problem,
 it is a fully coupled forward-backward stochastic differential equation,
 which is very difficult, if not possible, to be solved.
Meanwhile, it is natural to link the stochastic LQ problem with  Riccati equation. In  section 4, we will introduce two Riccati equation  and establish
its connection with the stochastic Hamilton system , and then prove the optimal control  has state feedback representation.

\section{ Existence and Uniqueness of Optimal Control}

In this section, we study the existence and uniqueness of 
the optimal control of Problem \ref{pro:1.1}.  To this end,  We first
establish
some elementary properties of the
cost functional.
\begin{lem}\label{lem:6.2}
  Let Assumptions  \ref{ass:1.1} and \ref{ass:1.2}
  be satisfied. Then  the cost functional $J(x,u(\cdot))$
is continuous over $\cal A.$
\end{lem}

\begin{proof}
 Let $(u (\cdot),  X(\cdot))$ and $(\bar u (\cdot), \bar X (\cdot))$ be any two admissible control pairs.
 Under Assumptions \ref{ass:1.1} and \ref{ass:1.2}, from  the definition of the cost functional $J(x,u(\cdot))$ (see \eqref{eq:1.2}),
 we  have 
\begin{eqnarray}\label{eq:5.10}
| J (u (\cdot)) - J (\bar u(\cdot) ) |^2
&\leq& K \bigg \{   \mathbb E\bigg[\int_0^T
  |u(t)-\bar u(t)|^2dt\bigg]
+   \mathbb E\bigg[\int_0^T
  |X(t)-\bar X(t)|^2dt\bigg] \bigg \} \\
&& \times \bigg \{  \mathbb E\bigg[\int_0^T
  |u(t)|^2dt\bigg]
+   \mathbb E\bigg[\int_0^T
  |X(t)|^2dt\bigg]
  + \mathbb E\bigg[\int_0^T
  |\bar u(t)|^2dt\bigg]
+   \mathbb E\bigg[\int_0^T
  |\bar X(t)|^2dt\bigg] \bigg \} \ . \nonumber
\end{eqnarray}
 Using the estimates \eqref{eq:1.4} and \eqref{eq:1.6} lead to 
\begin{eqnarray}
| J (u (\cdot)) - J (v (\cdot)) |^2
&\leq& K  \bigg \{   \mathbb E\bigg[\int_0^T
  |u(t)-\bar u(t)|^2dt\bigg] \bigg \} \times \bigg \{  \mathbb E\bigg[\int_0^T
  |u(t)|^2dt\bigg]
  + \mathbb E\bigg[\int_0^T
  |\bar u(t)|^2dt\bigg]+x \bigg \}  \ .
\end{eqnarray}
Thus, it follows that 
\begin{eqnarray}
J (u (\cdot)) - J (\bar u (\cdot)) \rightarrow 0 \ , \quad  as \quad u (\cdot) \rightarrow   \bar  u  (\cdot)
\quad in \quad {\cal A} \ .
\end{eqnarray}
The proof is complete.

\end{proof}

 \begin{lem}\label{lem:6.3}
   Let 
Assumptions \ref{ass:1.1} and \ref{ass:1.2}
  be satisfied. Then the cost functional $J(x,u(\cdot))$ is
strictly convex   $\cal A.$ Moreover, the cost
functional $J(x,u(\cdot))$ is coercive over $\cal A,$
i.e.,
$$\displaystyle\lim_ {\|u(\cdot)\|_{\cal A}{\rightarrow
\infty}}J(x,u(\cdot))=\infty.$$
\end{lem}
\begin{proof}
Since   the weighting matrices in the cost
functional is not random, it is  easy  to  check that
  \begin{eqnarray}\label{eq:2.4}
\begin{split}
 J(x, u(\cdot))=& \displaystyle \mathbb E\bigg[\int_0^T\bigg(\langle Q(t)(X(t)-\mathbb E[X(t)]),
X(t)-\mathbb E[X(t)])\rangle + \langle  (Q+\bar {
Q})(t)\mathbb E[X(t)], \mathbb E[X(t)]\rangle
 \\&+\langle N(t)(u(t)-\mathbb E[u(t)]), u(t)-\mathbb E [u(t)]\rangle
 +\langle  (N(t)+\bar
{N}(t))\mathbb E[u(t)], \mathbb E[u(t)]\rangle \bigg )dt\bigg]
\\&+\mathbb E\bigg[\langle
M(X(T)-\mathbb E [X(T)], X(T)-\mathbb E [X(T)]\rangle +\langle  (M+\bar{ M})\mathbb E[X(T)], \mathbb
E[X(T)]\rangle \bigg],
\end{split}
\end{eqnarray}
Thus  the cost functional $J(x,u(\cdot))$ over
$\cal A$ is convex from the nonnegativity of the
 $N, N+\bar N, Q, Q+\bar Q, M, M+\bar M $. Actually, since 
 $N, N+\bar N$ is uniformly positive, $J(u(\cdot))$ is strictly
convex. On the other hand, by Assumption \ref{ass:1.2}
and \eqref{eq:2.4} , we get
\begin{eqnarray}\label{eq:2.5}
  \begin{split}
    J(x, u(\cdot)) \geq& \mathbb E\bigg[\int_0^T
    \bigg (\langle N_1(t)(u(t)-\mathbb E[u(t)]), u(t)-\mathbb E[u(t)]\rangle
 +\langle  (N_1(t)+
{N}_2(t))\mathbb E[u(t)], \mathbb E[u(t)]\rangle \bigg )dt\bigg]
\\ \geq &\delta \mathbb E\bigg[\int_0^T\langle u(t)-\mathbb E[u(t)], u(t)-\mathbb E[u(t)]\rangle dt\bigg]+
 \delta\mathbb E\bigg[\int_0^T \langle \mathbb E[u(t)], \mathbb E[u(t)]\rangle dt\bigg]
 \\=& \delta\mathbb E \bigg[\int_0^T |u(t)|^2dt\bigg]
\\=&\delta ||u(\cdot)||^2_{\cal A}.
  \end{split}
\end{eqnarray}
Thus $\displaystyle\lim_ {\|u(\cdot)\|_{\cal A}{\rightarrow \infty}}J(x,u(\cdot))=\infty.$ The proof is complete.
\end{proof}

 \begin{lem}\label{lem:2.3}
Let
Assumptions \ref{ass:1.1} and \ref{ass:1.2}
be satisfied.  Then  the cost functional  $J(x, u(\cdot))$ is
Fr\`{e}chet differentiable over $\cal A$ and 
 the corresponding  Fr\`{e}chet
derivative $J'(x, u(\cdot))$ is  given by
\begin{eqnarray}\label{eq:2.6}
\begin{split}
 \langle J'(x, u(\cdot)),  v(\cdot) \rangle=&2\mathbb
E\bigg[\int_0^T\bigg(\langle Q(t)X^{(x,u)}(t),  X^{(0,v)}(t)\rangle
 +\langle \bar Q(t)\mathbb E [X^{(x,u)}(t)],  \mathbb E[X^{(0,v)}(t)]\rangle
\\&+\langle N(t)u(t), v(t)\rangle+\langle \bar N(t)\mathbb E[u(t)], \mathbb E [v(t)]\rangle \bigg)dt\bigg]
+2\mathbb E\bigg[\langle M X^{(x,u)}(T),  X^{(0,v)}(T)\rangle \\&+2\langle
\bar M\mathbb E[X^{(x,u)}(T)],  \mathbb E[X^{(0,v)}(T)]\rangle\bigg] , ~~~~\forall
u(\cdot), v(\cdot)\in{\cal A},
\end{split}
\end{eqnarray}
where $X^{(0,v)}(\cdot)$ is the solution of the state
equation  \eqref{eq:1.1}
corresponding to the admissible control $v(\cdot)$ and the initial
value $X(0)=0,$ and $X^{(x,u)}(\cdot)$ is the state process
corresponding to the control process $u(\cdot)$
with the initial
value $X(0)=x.$
\end{lem}

\begin{proof}

Let $u(\cdot)$ and $v(\cdot)$  be
two any given admissible control. For simplicity, the right hand side of \eqref {eq:2.6}
is denoted by  $\Delta^{u,v}.$
Since the state equation \eqref{eq:1.1} is linear,
it is easily to check that 
\begin{eqnarray}\label{eq:2.8}
  X^{x, u+v}(t)=X^{x, u}(t)+X^{0, v}(t), 0\leq t\leq T.
\end{eqnarray}
For simplicity, the right hand side of \eqref {eq:2.6}
is denoted by  $\Delta^{u,v}.$
 Therefore, in terms of  \eqref{eq:2.8}
 and  the definition  of the
cost functional $J(x, u(\cdot))$ (see \eqref{eq:1.2}), it is easy to check 
that
\begin{eqnarray} \label{eq:2.7}
  \begin{split}
 J(x, u(\cdot) +v(\cdot) )-J(x, u(\cdot) )= J(0,v(\cdot))+\Delta^{u,v}
  \end{split}
\end{eqnarray}
 On the other hand,  the
  estimate \eqref{eq:1.7} leads to
\begin{eqnarray}
  \begin{split}
    |J(0,v(\cdot))| \leq & K \mathbb E\bigg[\int_0^T|v(t)|^2dt\bigg]
  =K||v(\cdot) ||^2_{\cal A},
  \end{split}
\end{eqnarray}
Therefore,
\begin{eqnarray}
  \begin{split}
  \displaystyle\lim_ {\|v(\cdot) \|_{\cal A} {\rightarrow
0}}\frac{|J(x, u(\cdot) +v(\cdot) )-J(x, u(\cdot) )-\Delta^{u,v}|}{||v(\cdot) ||_{\cal A}}=\displaystyle\lim_ {\|v(\cdot) \|_{\cal A} {\rightarrow
0}}\frac{|J(0,v(\cdot)|}{||v(\cdot) ||_{\cal A}}
=0
  \end{split}
\end{eqnarray}
which gives  that $J(x, u(\cdot))$ has Fr\'{e}chet derivative  $\Delta^{u,v}.$
The proof is complete.
\end{proof}

\begin{rmk}
  Since the cost function $J(x, u(\cdot))$
  is Fr\'{e}chet differentiable,  then
  it is also G\^{a}teaux  differentiable.
  Moreover,the G\^{a}teaux  derivative
  is the  Fr\'{e}chet derivative
  $\langle J'(u(\cdot)),  v(\cdot) \rangle.$
  In fact,  from \eqref{eq:2.7}, we have
  \begin{eqnarray} \label{eq:2.10}
  \begin{split}
  &\displaystyle\lim_ {\eps {\rightarrow
0}}\frac{J(x, u(\cdot) +\eps v(\cdot) )-J(x, u(\cdot) )}{\eps}
\\=&
\displaystyle\lim_ {\eps {\rightarrow
0}}\frac{J(0,\eps v(\cdot))+\Delta^{u,\eps v}}{\eps}
\\=&
\displaystyle\lim_ {\eps {\rightarrow
0}}\frac{\eps^2J(0, v(\cdot))+\eps\Delta^{u, v}}{\eps}
\\=&\Delta^{u, v}
\\=&
\langle J'(u(\cdot)),  v(\cdot) \rangle
  \end{split}
\end{eqnarray}
\end{rmk}

Now  by  Lemma \ref{lem:6.2}-\ref{lem:2.3}, we can obtain the existence and uniqueness of optimal control.  This
result   is stated  as follows.
\begin{thm}\label{them:b1} 
Let Assumptions \ref{ass:1.1} and \ref{ass:1.2}
be satisfied. Then Problem \ref{pro:1.1} has a unique
optimal control. \end{thm}

\begin{proof}

Since the admissible controls set ${\cal A}=
 M^2_{\mathscr F}(0,T;\mathbb R^m)$ is a reflexive Banach space, in terms of  Lemma \ref{lem:6.2}-\ref{lem:2.3}, the uniqueness and existence of the optimal control of Problem \ref{pro:1.1}
 can be directly got from Proposition 2.12 of [7]
(i.e., the coercive, strictly convex and lower-semi continuous functional defined
on the reflexive Banach space has a unique 
minimum value point.)
 The proof is complete.
\end{proof}

\begin{thm}\label{thm:2.5}
 Let Assumptions \ref{ass:1.1} and \ref{ass:1.2}
be satisfied. Then  a necessary and sufficient condition for an admissible control
 $u(\cdot)\in \cal A$
 to be an optimal control  of Problem \ref{pro:1.1}  is that
  for any admissible control $v(\cdot) \in {\cal A},$
  \begin{equation}\label{eq:b16}
    \langle  J'(x,u(\cdot) ), v(\cdot) \rangle = 0,
  \end{equation}
  i.e.
  \begin{eqnarray}\label{eq:2.13}
\begin{split}
 0=&2\mathbb
E\bigg[\int_0^T\bigg(\langle Q(t)X^{(x,u)}(t),  X^{(0,v)}(t)\rangle
 +\langle \bar Q(t)\mathbb E[X^{(x,u)}(t)],  \mathbb E[X^{(0,v)}(t)]\rangle
\\&+\langle N(t)u(t), v(t)\rangle+\langle \bar N(t)\mathbb E[u(t)], \mathbb Ev(t)\rangle \bigg)dt\bigg]
+2\mathbb E\bigg[\langle MX^{(x,u)}(T),  X^{(0,v)}(T)\rangle \\&+2\langle
\bar M\mathbb E[X^{(x,u)}(T)],  \mathbb E[X^{(0,v)}(T)]\rangle\bigg] , ~~~~\forall
u(\cdot), v(\cdot)\in{\cal A}.
\end{split}
\end{eqnarray}
\end{thm}

\begin{proof}
  For the necessary part,  suppose that $u(\cdot)$
  is  an optimal control.
   Then  from \eqref{eq:2.10},  for any admissible control $v(\cdot),$ we have
  \begin{eqnarray}
  \begin{split}
 \langle J'(x,u(\cdot)),  v(\cdot) \rangle= \displaystyle\lim_ {\eps {\rightarrow
0^{+}}}\frac{J(x, u(\cdot) +\eps v(\cdot) )-J(x, u(\cdot) )}{\eps}
\geq 0,
  \end{split}
\end{eqnarray}
and
\begin{eqnarray}
  \begin{split}
  -\langle J'(x,u(\cdot)),  v(\cdot)\rangle=\langle J'(x,u(\cdot)),  -v(\cdot) \rangle= \displaystyle\lim_ {\eps {\rightarrow
0^{+}}}\frac{J(x, u(\cdot) +\eps (-v(\cdot)) )-J(x, u(\cdot) )}{\eps}
\geq 0,
  \end{split}
\end{eqnarray}
which imply  that
 \begin{equation}\label{eq:b16}
    \langle  J'(u(\cdot) ), v(\cdot) \rangle = 0.
  \end{equation}
For the sufficient part, let $u(\cdot)$ be an given
admissible control, and suppose that
for any admissible control $v(\cdot),$
$ \langle  J'(u(\cdot) ), v(\cdot) \rangle = 0.$
Since the cost functional $J$
is convex, then we have
\begin{equation}\label{eq:b16}
 J(x,v(\cdot))-J(x,u(\cdot))  \geq  \langle  J'(x,u(\cdot) ), v(\cdot) \rangle = 0.
  \end{equation}
which
implies that
$u(\cdot)$ is an optimal
control. The proof is complete.
\end{proof}

\section {Optimality Conditions and Stochastic Hamilton Systems}
Now we derive
  a necessary  and  sufficient condition  for an  admissible pair of
 Problem  \ref{pro:1.1} to be
  an optimal pair by  adjoint equation.
\begin{thm}\label{thm:b2}
Let  Assumptions \ref{ass:1.1} and
\ref{ass:1.2}  be satisfied.
 Then, a necessary and
sufficient condition for an admissible pair $(u(\cdot); X(\cdot))$ to be an optimal pair of  Problem  is that  the admissible
control $u(\cdot)$ satisfies
\begin{eqnarray} \label{eq:3.1000}
  \begin{split}
 &2N(t)u(t)+2\bar N(t)\mathbb E[u(t)]+B^\top(t)p({t-})+\bar B^\top(t) \mathbb E [p({t-})]+ D^{\top}(t)q(t)
 +\bar D^{\top}(t) \mathbb E [q(t)]
\\&~~~~~~~~~~~+\int_ ZF^\top(t,\theta)r(t,\theta)\nu (d\theta)+\int_ Z
\bar F^\top(t,\theta) \mathbb E [r(t,\theta)]\nu (d\theta)=0, \quad a.e. a.s.,
     \end{split}
\end{eqnarray}
where $(p(\cdot),q(\cdot), r(\cdot, \cdot)) $ is the unique
solution of the following  mean-field backward SDE
(MF-BSDE)
\begin {equation}\label{eq:3.2}
\left\{\begin{array}{lll}
dp(t)&=&-\bigg[A^\top(t)p(t)+\bar  A(t)^\top\mathbb E [p(t)]+C^{\top}(t)q(t)+\bar  C^{\top}(t)\mathbb E[q(t)]+\displaystyle\int_ ZE^\top(t,\theta)r(t,\theta)\nu (d\theta)\\&&+\displaystyle\int_ Z \bar E^\top(t,\theta)\mathbb E[r(t,\theta)]\nu
(d\theta)+2Q(t)X(t)+2\bar Q(t)\mathbb E[X(t)]\bigg]dt
+q(t)dW(t)+\displaystyle\int_ Z r(t,\theta)\tilde{\mu}(d\theta, dt),
 \\p(T)&=&2GX(T)+2\bar G\mathbb E[X(T)].
\end{array}
  \right.
  \end {equation}

\end{thm}
\begin{proof}
 Let $u(\cdot)\in \cal A$
 be a given  admissible  control.
 Then for any admissible control $v(\cdot) \in {\cal A},$ from Lemma \ref{lem:2.3}, we have

\begin{eqnarray}\label{eq:3.3}
\begin{split}
 \langle J'(u(\cdot)),  v(\cdot) \rangle=&2\mathbb
E\bigg[\int_0^T\bigg(\langle Q(t)X^{(x,u)}(t),  X^{(0,v)}(t)\rangle
 +\langle \bar Q(t)\mathbb E[X^{(x,u)}(t)],  \mathbb E[X^{(0,v)}(t)]\rangle
\\&+\langle N(t)u(t), v(t)\rangle+\langle \bar N(t)\mathbb E[u(t)], \mathbb E[v(t)]\rangle \bigg)dt\bigg]
+2\mathbb E\big[\langle MX^{(x,u)}(T),  X^{(0,v)}(T)\rangle\big] \\&+2\langle
\bar M\mathbb E[X^{(x,u)}(T)],  \mathbb E[X^{(0,v)}(T)]\rangle.
\end{split}
\end{eqnarray}
On the other hand, by [23], we know that
\eqref{eq:3.2}
 admits a unique adapted solution $(p(\cdot),q(\cdot), r(\cdot, \cdot)) $ . Applying Itô's formula to
$\langle X^{0, u}(t), q(t)  \rangle$
 and taking expectation, we
have
\begin{eqnarray}\label{eq:3.4}
\begin{split}
&2\mathbb
E\bigg[\int_0^T\bigg(\langle Q(t)X^{(x,u)}(t),  X^{(0,v)}(t)\rangle
 +\langle \bar Q(t)\mathbb E[X^{(x,u)}(t)],  \mathbb E[X^{(0,v)}(t)]\rangle \bigg)\bigg]dt
\\&~~~~~~+2\mathbb E\bigg[\langle GX^{(x,u)}(T),  X^{(0,v)}(T)\rangle+2\langle
\bar G\mathbb E[X^{(x,u)}(T)],  \mathbb E[X^{(0,v)}(T)]\rangle \bigg]
\\ =&\mathbb
E\bigg [\int_0^T
\bigg(\langle p(t), B(t)v(t)+\bar B(t)\mathbb E [v(t)]\rangle+
\langle q(t), D(t)v(t)+\bar D(t)\mathbb E[v(t)]
\\& \quad \quad +\int_{z}\langle r(t,\theta), F(t,\theta)v(t)+\bar F(t,\theta)\mathbb E[v(t)]\rangle \nu d(\theta)\bigg)dt\bigg ]
\\=& \mathbb
E\bigg[\int_0^T \bigg \langle B^\top (t)p({t-})+\bar B^\top(t) \mathbb E [p({t-})]+ D^{\top}(t)q(t)
 +\bar D^{\top}(t) \mathbb E [q(t)]
\\&~~~~~~~~~~~+\int_ ZF^\top(t,\theta)r(t,\theta)\nu (d\theta)+\int_ Z
\bar F^\top(t,\theta) \mathbb E [r(t,\theta)]\nu (d\theta), v(t)\bigg\rangle dt\bigg]
\end{split}
\end{eqnarray}
Putting \eqref{eq:3.4} into \eqref{eq:3.3}, we get
\begin{eqnarray}\label{eq:4.5}
\begin{split}
 &\mathbb
\langle J'(u(\cdot)),  v(\cdot) \rangle=
\mathbb E\bigg[\int_0^T \bigg \langle 2N(t)u(t)+2\bar N(t)\mathbb E[u(t)] +B^\top(t)+p({t-})+\bar B^\top(t) \mathbb E [p({t-})]+ D^{\top}(t)q(t)
 \\&~~~~~~~~~~~+\bar D^{\top}(t) \mathbb E[ q(t)]
+\int_ ZF^\top(t,\theta)r(t,\theta)\nu (d\theta)+\int_ Z
\bar F^\top(t,\theta) \mathbb E [r(t,\theta)]\nu (d\theta), v(t)\bigg\rangle dt\bigg]
\end{split}
\end{eqnarray}
For the necessary,  let $(u(\cdot); X(\cdot))$ be an optimal pair,  then from Theorem  \ref{thm:2.5},  we have $ \langle J'(u(\cdot)),  v(\cdot) \rangle=0$
which imply that
\begin{eqnarray} \label{eq:3.60}
  \begin{split}
 &2N(t)u(t)+2\bar N(t)\mathbb E[u(t)]+B^\top (t)p({t-})+\bar B^\top (t) \mathbb E [p({t-})]+ D^{\top}(t)q(t)
 +\bar D^{\top}(t) \mathbb E[ q(t)]
\\&~~~~~~~~~~~~+\int_ ZF^\top (t,\theta)r(t,\theta)\nu (d\theta)+\int_ Z
\bar F^\top (t,\theta) \mathbb E[r(t,\theta)]\nu (d\theta)=0, a.e. a.s.,
     \end{split}
\end{eqnarray} from \eqref{eq:4.5},  since
$v(\cdot)$ is arbitrary.

For the sufficient part, let $(u(\cdot); X(\cdot))$ be an admissible pair satisfying \eqref{eq:3.1000}.
Putting \eqref{eq:3.1000} into \eqref{eq:4.5},
then we have  $ \langle J'(u(\cdot)),  v(\cdot) \rangle=0$
which implies that $(u(\cdot); X(\cdot))$ is an optimal control
pair from Theorem \ref{thm:2.5}.
\end{proof}
\vspace{1mm}

Finally we introduce the so-called stochastic Hamilton system which
consists of the state equation \eqref{eq:1.1}, the adjoint equation
\eqref{eq:3.2}  and the dual
representation \eqref{eq:3.1000}:

\begin{numcases}{}\label{eq:3.7}
dX(t)=[A(t)X(t)+\bar A(t)\mathbb E[X(t)]
  +B(t)u(t)+\bar B(t)\mathbb E[u(t)]]dt \nonumber
  \\ \quad \quad \quad \quad +[C(t)X(t)
  +\bar C(t)\mathbb E[X(t)]
  +D(t)u(t)+\bar D(t)\mathbb E[u(t)]]dW(t) \nonumber
  \\~~~~~~~~~~~~~+\displaystyle
  \int_{Z}[E(t,\theta)X({t-})+\bar E(t,\theta)\mathbb E [X({t-})]
  +F(t,\theta)u(t)+\bar F(t, \theta)\mathbb E[u(t)]]\tilde{\mu}(d\theta,
  dt), \\ \nonumber
dp(t)=-\big[A^\top (t)p(t)+\bar  A^\top(t)\mathbb E [p(t)]+C^{\top }(t)q(t)+\bar  C^{\top}(t)\mathbb E [q(t)]+\displaystyle\int_ ZE^\top(t,\theta)r(t,\theta)\nu (d\theta) \nonumber
\\~~~~~~~~~~~~+\displaystyle\int_ Z \bar E^\top(t,\theta)\mathbb E[r(t,\theta)]\nu
(d\theta)+2Q(t)X(t)+2\bar Q(t)\mathbb E[X(t)]\big]dt
+q(t)dW(t)+\displaystyle\int_ Z r(t,\theta)\tilde{\mu}(d\theta, dt),
\nonumber
 \\X(0)=x,\quad p(T)=2GX(T)+2\bar G\mathbb E[X(T)], \nonumber
 \\ 2N(t)u(t)+2\bar N(t)\mathbb E[u(t)]+B^\top(t)p({t-})+\bar B^\top(t) \mathbb E [p({t-})]+ D^{\top}(t)q(t)
 +\bar D^{\top}(t) \mathbb E[ q(t)]
+\int_ ZF^\top(t,\theta)r(t,\theta)\nu (d\theta) \nonumber
\\\quad+\int_ Z
\bar F^\top(t,\theta) \mathbb E [r(t,\theta)]\nu (d\theta)
=0. \nonumber
\end{numcases}

 This is a fully coupled mean-field forward- backward stochastic differential equation
 (MF-FBSDE in short) and its solution consists of $( u(\cdot), X(\cdot), p(\cdot),q(\cdot), r(\cdot, \cdot))$.

\begin{thm} \label{thm:6.7} 
Let Assumptions \ref{ass:1.1} and \ref{ass:1.2}
be satisfied. Then 
stochastic Hamilton system \eqref{eq:3.7} has a unique solution
$(u(\cdot), X(\cdot),p(\cdot),
 q(\cdot), r(\cdot))\in M_{\mathscr{F}}^2(0,
T;\mathbb R^m)\times S_{\mathscr{F}}^2(0,
T;\mathbb R^n)\times S_{\mathscr{F}}^2(0, T;\mathbb R^n)\times
{M}_{\mathscr{F}}^{\nu,2}{([0,T]\times  Z; 
\mathbb R^n)}.$  And $u(\cdot)$ is the unique optimal
control of Problem \ref{pro:1.1} and $X(\cdot)$ is its
corresponding optimal state process.
\end{thm}
\begin{proof}  By Theorem \ref{them:b1},
  Problem \ref{pro:1.1} admits a unique optimal
pair $(u(\cdot),X(\cdot)).$  Suppose   $(p(\cdot),q(\cdot), r(\cdot, \cdot))$
is the unique
solution of the adjoint equation \eqref{eq:3.2} corresponding to  the optimal pair
$({u}(\cdot); {X}(\cdot))$.
Then by  the necessary part of Theorem \ref{thm:b2}, the optimal control has the dual presentation \eqref{eq:3.1000}. Consequently,
 $(u(\cdot), X(\cdot), p(\cdot),q(\cdot), r(\cdot, \cdot))$ consists of an adapted
solution to the  stochastic Hamilton system \eqref{eq:3.7}.
Next, if the stochastic Hamilton system \eqref{eq:3.7} has an
another adapted solution $(\bar u(\cdot),\bar X(\cdot), \bar p(\cdot),\bar q(\cdot), \bar r(\cdot, \cdot)),$  then $(\bar u(\cdot), \bar X(\cdot))$
must be an optimal pair of Problem \ref{pro:1.1} by the sufficient
part of Theorem \ref{thm:b2}. So we must have $  u(\cdot)= \bar
u(\cdot)$ by the  uniqueness of  the optimal control. Furthermore,
by the uniqueness of the solutions  of MF-SDE and  MF-BSDE, one
must have $( \bar X(\cdot), \bar p(\cdot), \bar q(\cdot), r(\cdot, \cdot))=(X(\cdot), p(\cdot),q(\cdot), r(\cdot, \cdot))$ .
The proof is complete.
\end{proof}
\begin{rmk}
In summary, the stochastic Hamilton system \eqref{eq:3.7}
completely characterizes the optimal control of Problem \ref{pro:1.1}. Therefore, solving Problem \ref{pro:1.1} is equivalent to solving
the stochastic Hamilton system, moreover, the unique optimal control
can be given by \eqref{eq:3.1000}.
Taking expectation to \eqref{eq:3.1000},  we have

\begin{eqnarray} \label{eq:3.8}
  \begin{split}
 &2(N(t)+\bar N(t))\mathbb E[u(t)]
 +(B^\top (t)+\bar B^\top (t)) \mathbb E [p({t-})]+ (D^{\top}(t)
 +\bar D^{\top}(t)) \mathbb E [q(t)]
\\&~~~~~~~~~~~~+\int_ Z(F^\top (t,\theta)+
\bar F^\top (t,\theta) )\mathbb E [r(t,\theta)]\nu (d\theta)=0, \quad a.e. a.s.,
     \end{split}
\end{eqnarray}
which implies that

\begin{eqnarray} \label{eq:3.9}
  \begin{split}
 &\mathbb E[u(t)]
 =-\frac{1}{2}(N(t)+\bar N(t))^{-1}\bigg[(B (t)+\bar B (t))^\top \mathbb E [p({t-})]+ (D(t)
 +\bar D(t))^\top \mathbb E [q(t)]
\\&~~~~~~~~~~~~+\int_ Z(F (t,\theta)+
\bar F (t,\theta) )^\top\mathbb E [r(t,\theta)]\nu (d\theta)\bigg], \quad a.e. a.s.
     \end{split}
\end{eqnarray}

From\eqref{eq:3.1000}, we know that
\begin{eqnarray} \label{eq:3.10}
  \begin{split}
 &2N(t)u(t)=-2\bar N(t)\mathbb E[u(t)]-B^\top (t)p({t-})-\bar B^\top (t) \mathbb E[ p({t-})]- D^{\top}(t)q(t)
 -\bar D^{\top}(t) \mathbb E [q(t)]
\\&~~~~~~~~~~~~-\int_ ZF^\top (t,\theta)r(t,\theta)\nu (d\theta)-\int_ Z
\bar F^\top (t,\theta) \mathbb E [r(t,\theta)]\nu (d\theta), \quad a.e. a.s.,
     \end{split}
\end{eqnarray}

Then putting \eqref{eq:3.9} into \eqref{eq:3.10},
we  have
\begin{eqnarray} \label{eq:3.6}
  \begin{split}
 u(t)=&-\frac{1}{2}N^{-1}(t)\bigg\{B^\top (t)p({t-})+\bar B^\top (t) \mathbb E [p({t-})]+ D^{\top}(t)q(t)
 +\bar D^{\top}(t) \mathbb E [q(t)]
\\&~~~~~~~~~~~~~~~+\int_ ZF^\top (t,\theta)r(t,\theta)\nu (d\theta)+\int_ Z
\bar F^\top (t,\theta) \mathbb E [r(t,\theta)]\nu (d\theta)\bigg]
\\&+\bar N(t)(N(t)+\bar N(t))^{-1}\bigg[(B (t)+\bar B (t))^\top \mathbb E [p({t-})]+ (D(t)
 +\bar D(t))^\top \mathbb E [q(t)]
\\&~~~~~~~~~~~~~~~~~~~~~\quad\quad\quad\quad+\int_ Z(F (t,\theta)+
\bar F (t,\theta) )^\top\mathbb E[r(t,\theta)]\nu (d\theta)\bigg]\bigg\}
, \quad a.e. a.s.
     \end{split}
\end{eqnarray}

\end{rmk}

\section{ Backward  Riccati equation and 
State Feedback Representation of Optimal Control }

Although the optimal control of 
 Problem \ref{pro:1.1} is completely characterized by the  stochastic Hamilton system \eqref{eq:3.7}, 
 \eqref{eq:3.7} is a fully coupled  mean-field forward-backward stochastic differential equation whose solvability is much difficult to be obtained.
Meanwhile, it is natural to link the stochastic LQ problem with Riccati equation. In this section, we will introduce
two Riccati equation  and establish
its connection with the stochastic Hamilton system \eqref{eq:3.7}, and then prove the optimal control has
 state feedback representation. 

\subsection{Derivation of   Riccati equations}

Let $(u(\cdot), X(\cdot))$ be the
optimal pair of Problem \ref{pro:1.1}
associated with the adjoint process  $(p(\cdot), q(\cdot), r(\cdot, \cdot))$
 being the solution of the  adjoint equation 
 \eqref{eq:3.2} .
This means $(u(\cdot), X(\cdot), p(\cdot), q(\cdot), r(\cdot, \cdot))$  is the solution of the stochastic  Hamilton system \eqref{eq:3.7}.

Taking expectation on both sides of  the stochastic  Hamilton system \eqref{eq:3.7},
we get that $(\mathbb E [X(\cdot)], \mathbb E [p(\cdot)])$ satisfies  the following forward-backward
ordinary differential equation (suppressing s)
\begin{equation}\label{eq:4.1}
\left\{\begin {array}{lll}
  d\mathbb E[X]&=&\bigg[(A+\bar A)\mathbb E [X]
  +(B+\bar B)\mathbb E [u])\bigg]dt,\\
  d\mathbb E[p]&=&-\bigg[(A^\top+\bar  A^\top)\mathbb E[p]+(C^{\top}+\bar  C^{\top})\mathbb E[q]+\displaystyle\int_ Z(E^\top(\theta)+ \bar E^\top(\theta))\mathbb E[r(t,\theta)]\nu
(d\theta)+2(Q+\bar Q)\mathbb E[X]\bigg]dt
 \\\mathbb E [p(T)]&=&2(G+\bar G)\mathbb E[X(T)].
   \\\mathbb E[X(0)]&=&x.
\end {array}
\right.
\end{equation}
Further,  it is easy to check that  $(X(\cdot)-\mathbb E [X(\cdot)], p(\cdot)-\mathbb E [p(\cdot)])$ satisfies   the following forward-backward
stochastic differential equation

\begin{equation}\label{eq:4.3}
\left\{\begin {array}{ll}
  d(X-\mathbb E[X])&=\bigg[A(X-\mathbb E[X])
  +B(u-\mathbb E[u])\bigg]dt+\bigg[   C(X-\mathbb E[X])
  + (C+\bar C)\mathbb E[X]
  +D(u-\mathbb E[u])
  + (D+\bar D)\mathbb E[u]\bigg]dW(t)
  \\&+\displaystyle
  \int_{Z}\bigg[E(\theta)(X-
   \mathbb E[X])+ (E(\theta)+\bar E(\theta))\mathbb E [X]
  +F(\theta)(u-\mathbb E [u])+ (F(\theta)+\bar F(\theta))\mathbb E [u]\bigg]\tilde{\mu}(d\theta,
  dt),\\
  d\big(p-\mathbb E[p])&=-\bigg[A^\top \big(p-\mathbb E [p])+C^{\top}\big(q-\mathbb E[q])+\displaystyle\int_ ZE^\top(t,\theta)(r(t,\theta)-\mathbb E [r(t,\theta))]\nu (d\theta)+2Q(X-\mathbb E[X])\bigg]dt
\\&+qdW(t)+\displaystyle\int_ Z r(t,\theta)\tilde{\mu}(d\theta, dt),
 \\p(T)-\mathbb E[ p(T)]&=2G(X(T)-\mathbb E[X(T)]),
   \\X(0)-\mathbb E[X(0)]&=0.
\end {array}
\right.
\end{equation}
In view of  the terminal condition of
   the equations \eqref{eq:4.1} and \eqref{eq:4.3},
  now we assume that the
  state equation $X(\cdot)$ and
  the adjoint equation
  $p(\cdot)$  have the following
  relationship:
\begin{eqnarray}
  \begin{split}
    p(t)=P(t)(X(t)-\mathbb E[X(t)])
    +\Pi(t)\mathbb  E[X(t)],
  \end{split}
\end{eqnarray}
where
$P(\cdot)$ and $\Pi(\cdot)$ taking values in
$S^n_{+}$ are  some deterministic  differentiable
functions such that
\begin{eqnarray}
  \begin{split}
    P(T)=G, \quad \quad  \Pi(T)=G+\bar G.
  \end{split}
\end{eqnarray}
Consequently,  we further  have the following
relationship
\begin{eqnarray} \label{eq:4.50}
  \begin{split}
    \mathbb E [p(t)]=\Pi(t)\mathbb  E[X(t)]
  \end{split}
\end{eqnarray}
 and

 \begin{eqnarray} \label{eq:4.6}
  \begin{split}
   p(t)-\mathbb E [p(t)]=P(t)(X(t)-\mathbb E[X(t)]).
  \end{split}
\end{eqnarray}
In the following,  we
begin to formerly derive the
corresponding Riccati equations
which $P(\cdot)$ and $\Pi(\cdot)$
should satisfy. From the relationship
\eqref{eq:4.6} and \eqref{eq:4.3},
applying It\^{o} formula
 to $P(t)(X(t)-\mathbb E[X(t)])$ leads to
\begin{eqnarray} \label{eq:4.7}
  \begin{split}
&-\bigg[A^\top\big(p-\mathbb E [p])+C^\top\big(q-\mathbb E [q])
+\displaystyle\int_ ZE^\top(\theta)(r(\theta)-\mathbb E [(r(\theta)])\nu (d\theta)+2Q(X-\mathbb E[X])\bigg]dt
\\&~~~~~~+qdW(t)+\displaystyle\int_ Z r(\theta)\tilde{\mu}(d\theta, dt)
\\&=d\big(p-\mathbb E[p])
\\&=dP(X-\mathbb E[X])
\\&= \bigg[\dot{P}(X-\mathbb E[X])+
P\bigg(A(X-\mathbb E[X])
  +B(u-\mathbb E [u])\bigg)\bigg]dt
  \\&\quad+P\bigg[   C(X-\mathbb E[X])
  + (C+\bar C)\mathbb E[X]
  +D(u-\mathbb E[u])
  + (D+\bar D)\mathbb E[u]\bigg]dW(t)
  \\& \quad+\displaystyle
  \int_{Z}P\bigg[E(\theta)(X-
   \mathbb E[X])+ (E(\theta)+\bar E(\theta))\mathbb E [X]+F(\theta)(u-\mathbb E [u])+ (F(\theta)+\bar F(\theta))\mathbb E [u]\bigg]\tilde{\mu}(d\theta,
  dt).
  \end{split}
\end{eqnarray}
Comparing the diffusion terms of
both sides of the above equality, we have
\begin{eqnarray} \label{eq:4.8}
  \begin{split}
    q=P\bigg[   C(X-\mathbb E[X])
  + (C+\bar C)\mathbb E[X]
  +D(u-\mathbb E[u])
  + (D+\bar D)\mathbb E[u]\bigg]
  \end{split}
\end{eqnarray}
and

\begin{eqnarray}\label{eq:4.9}
  \begin{split}
    r(\theta)=P\bigg[E(\theta)(X- \mathbb [ X])+ (E(\theta)+\bar E(\theta))\mathbb E[X]
  +F(\theta)(u-\mathbb E [u])+ (F(\theta)+\bar F(\theta))\mathbb E [u]\bigg].
  \end{split}
\end{eqnarray}
Then  taking expectation on both sides of
\eqref{eq:4.8} and \eqref{eq:4.9}, 
we have the following relationships:

\begin{eqnarray} \label{eq:4.10}
  \begin{split}
    \mathbb E [q]=P\bigg[   (C+\bar C)\mathbb E[X]
  + (D+\bar D)\mathbb E[u]\bigg],
  \end{split}
\end{eqnarray}
\begin{eqnarray} \label{eq:4.11}
  \begin{split}
    \mathbb E [r(\theta)]=P\bigg[(E(\theta)+\bar E(\theta))\mathbb E [X]+ (F(\theta)+\bar F(\theta))\mathbb E [u]\bigg]
  \end{split}
\end{eqnarray}

\begin{eqnarray} \label{eq:4.12}
  \begin{split}
    q-\mathbb E [q]=P\bigg[   C(X-\mathbb E[X])
  +D(u-\mathbb E[u])\bigg],
  \end{split}
\end{eqnarray}

\begin{eqnarray} \label{eq:4.13}
  \begin{split}
    r(\theta)-\mathbb E [r(\theta)]=P\bigg[E(\theta)(X- E[\bar X])
  +F(\theta)(u-\mathbb E[u])\bigg].
  \end{split}
\end{eqnarray}
In view of  \eqref{eq:3.1000}
and \eqref{eq:3.8}, we get that
\begin{eqnarray} \label{eq:4.14}
  \begin{split}
 &2N(u-\mathbb E[u])
 +B^\top(p-\mathbb E[p])+ D^{\top}(q-\mathbb E[q])
\\&~~~~~~~~~~~+\int_ ZF^\top(\theta)(r(\theta)-\mathbb E[ r(\theta)])\nu (d\theta)=0, \quad a.e. a.s.
     \end{split}
\end{eqnarray}
Then putting  \eqref{eq:4.6},\eqref{eq:4.12} and
\eqref{eq:4.13}  into \eqref{eq:4.14} yields
\begin{eqnarray}
 \begin{split}
0=&N (u-\mathbb E[u]) +B^\top P(X-\mathbb E][X])
+ D^{\top}P\bigg[   C(X-\mathbb E[X])
  +D(u-\mathbb E[u])\bigg]
\\&+\int_ Z(F^\top (\theta)P\bigg[E(\theta)(X- \mathbb E[X])
  +F(\theta)(u-\mathbb E [u])\bigg]\nu (d\theta)
 \\ =&\bigg[N+ D^{\top}PD+\int_ ZF^\top(\theta)PF(\theta)\nu (d\theta)\bigg]
 (u-\mathbb E[u])
 \\& +\bigg[B^\top P
+ D^{\top}PC+\int_ ZF^\top (\theta)PE(\theta)\nu (d\theta)\bigg](X-\mathbb E[X])
\\=&\Sigma_0(u-\mathbb E[u])
 +\bigg[B^\top P
+ D^{\top }PC+\int_ ZF^\top (\theta)PE(\theta)\nu (d\theta)\bigg](X-\mathbb E[X]),
 \end{split}
\end{eqnarray}
where we denote
$$\Sigma_0=N+ D^{\top}PD+\int_ ZF^\top(\theta)PF(\theta)\nu (d\theta).$$
This implies that
\begin{eqnarray} \label{eq:4.16}
 \begin{split}
u-\mathbb E[u]
 =-\Sigma_0^{-1}\bigg[B^TP
+ D^{T}PC+\int_ Z(F^T(\theta)PE(\theta)\nu (d\theta)\bigg]\bigg[X-\mathbb E[X]\bigg].
 \end{split}
\end{eqnarray}
Comparing the drift terms in both
 sides of \eqref{eq:4.7} and
 combining \eqref{eq:4.6},
 \eqref{eq:4.12},
 \eqref{eq:4.13} and \eqref{eq:4.16}, we get that
\begin{eqnarray}
  \begin{split}
   0=& (\dot{P}+PA)(X-\mathbb E[X]) + PB(u-\mathbb E[ u])
   +A^\top\big(p-\mathbb E[p])
   \\&+C^{\top}\big(q-\mathbb E [q])
+\displaystyle\int_ ZE^\top(\theta)(r(\theta)-\mathbb E [r(\theta)])\nu (d\theta)+2Q (X-\mathbb E[X])
\\=& (\dot{P}+PA)(X-\mathbb E[X])+  PB(u-\mathbb E[u])
    \\&+A^\top P(X-\mathbb E[X])+C^{\top}P\bigg [ C(X-\mathbb E[X])
  +D(u-\mathbb E[u])\bigg]
\\&+\displaystyle\int_ ZE^\top(\theta)P\bigg[E(\theta)(X- \mathbb E[ X])
  +F(\theta)(u-\mathbb E [u])\bigg]\nu (d\theta)+2Q (X-\mathbb E[X])
\\=& \bigg[(\dot{P}+PA+A^\top P+C^{T}PC
+\displaystyle\int_ ZE^\top (\theta)PE(\theta)\nu (d\theta)+2Q\bigg](X- E[ X])
\\&+\bigg(PB
+ C^{\top}PD+\int_ Z(E^\top(\theta)PF(\theta)\nu (d\theta)\bigg)(u-\mathbb E[u])
\\=& \bigg[(\dot{P}+PA+A^\top P+C^{\top}PC
+\displaystyle\int_ ZE^\top(\theta)PE(\theta)\nu (d\theta)+2Q\bigg]\bigg[X- \mathbb E [X]\bigg]
\\&-\bigg[PB
+ C^{\top}PD+\int_ ZE^\top(\theta)PF(\theta)\nu (d\theta)\bigg]\Sigma_0^{-1}\\
&\cdot\bigg[B^TP
+ D^{\top}PC+\int_ ZF^\top(\theta)PE(\theta)\nu (d\theta)\bigg]\bigg[X-\mathbb E[X]\bigg].
  \end{split}
\end{eqnarray}
Therefore we should let  $P(\cdot)$
be the solution to the following
Riccati equation
\begin {equation}\label{eq:4.18}
\left\{\begin{array}{lll}
 &(\dot{P}+PA+A^\top P+C^{\top}PC
+\displaystyle\int_ ZE^\top (\theta)PE(\theta)\nu (d\theta)+2Q
\\&~~~~~~~-\bigg [PB
+ C^{\top}PD+ \displaystyle\int_ ZE^\top(\theta)PF(\theta)\nu (d\theta)\bigg]\Sigma_0^{-1}
\\&~~~~~~~~~~\cdot\bigg [B^\top P(s)
+ D^{\top}PC+\displaystyle\int_ ZF^\top (\theta)PE(\theta)\nu (d\theta)\bigg]
=0,
\\& P(T)=G.
 \end{array}
  \right.
  \end {equation}
On the other hand,
putting
\eqref{eq:4.50},\eqref{eq:4.10} and \eqref{eq:4.11} into  \eqref{eq:3.8},
 we get that
\begin{eqnarray}
  \begin{split}
   0= &2(N+\bar N)\mathbb E[u]+(B^\top
    +\bar B^\top) \mathbb E [p]+ (D^{\top}
 +\bar D^{\top}) \mathbb E[q]+\int_ Z(F^\top(\theta)+\bar F^\top(\theta))\mathbb E[ r(\theta)]\nu (d\theta)
 \\=&2(N+\bar N)\mathbb E[u]+(B^\top
    +\bar B^\top) \Pi\mathbb  E[X]+ (D^{\top}
 +\bar D^{\top}) P\bigg[   (C+\bar C)\mathbb E[X]
  + (D+\bar D)\mathbb E[u]\bigg]\\
 &\quad\quad+\int_ Z(F^\top(\theta)+\bar F^\top(\theta))P\bigg[(E(\theta)+\bar E(\theta))\mathbb E[X]+ (F(\theta)+\bar F(\theta))\mathbb E [u]\bigg]
 \\=&\bigg[2(N+\bar N)+(D^{\top}
 +\bar D^{\top}) P(D+\bar D)+\int_ Z(F^\top(\theta)+\bar F^\top(\theta))P(F(\theta)+\bar F(\theta))\bigg]\mathbb E[u]
 \\&\quad\quad+\bigg[(B^\top
    +\bar B^\top) \Pi(s)+ (D^{\top}
 +\bar D^{\top}) P(C+\bar C)+\int_ Z(F^T(\theta)+\bar F^\top(\theta))P(E(\theta)+\bar E(\theta))\bigg]\mathbb E[X]
 \\=&\Sigma_2\mathbb E[u]+\bigg[(B^\top
    +\bar B^\top) \Pi+ (D^{T}
 +\bar D^{\top}) P(C+\bar C)+\int_ Z(F^\top(\theta)+\bar F^\top(\theta))P(E(\theta)+\bar E(\theta))\bigg]\mathbb E[X],
  \end{split}
\end{eqnarray}
where
\begin{eqnarray}
  \begin{split}
   \Sigma_2:= 2(N+\bar N)+(D^{\top}
 +\bar D^{\top}) P(D+\bar D)+\int_ Z(F^\top(\theta)+\bar F^\top(\theta))P(F(\theta)+\bar F(\theta)).
  \end{split}
\end{eqnarray}
This implies that
\begin{eqnarray}\label{eq:4.21}
  \begin{split}
   \mathbb E[u]=-\Sigma_2^{-1}\bigg[(B^\top
    +\bar B^\top) \Pi+ (D^{\top}
 +\bar D^{\top}) P(C+\bar C)+\int_ Z(F^\top (\theta)+\bar F^\top(\theta))P(E(\theta)+\bar E(\theta))\bigg]\mathbb E[X].
  \end{split}
\end{eqnarray}
Furthermore,  from 
\eqref{eq:4.50} and \eqref{eq:4.1}, we have

\begin{eqnarray}
  \begin{split}
  &-\bigg[(A^\top+\bar  A^\top )\mathbb E [p]+(C^{\top}+\bar  C^{\top})\mathbb E[q(t)]+\displaystyle\int_ Z(E^\top(\theta)+ \bar E^\top(\theta))\mathbb E[r(\theta)]\nu
(d\theta)+2(Q+\bar Q)\mathbb E[X]\bigg]dt
   \\=&d\mathbb E[ p]
   \\=&d\Pi\mathbb E[X]
   \\=&\bigg[\dot{\Pi}\mathbb E[X]+\Pi(A+\bar A)\mathbb E[X]
  +\Pi(B+\bar B)\mathbb E [u])\bigg]dt.
  \end{split}
\end{eqnarray}
Putting \eqref{eq:4.50}, \eqref{eq:4.10}
 and \eqref{eq:4.11} into the left hand
 of the above equality
 and comparing both sides  of the above equality, we get

 \begin{eqnarray} \label{eq:4.23}
  \begin{split}
  0=&\dot{\Pi}\mathbb E[X]+\Pi(A+\bar A)\mathbb E[X]
  +(B+\bar B)\mathbb E [u]+
  (A^T+\bar  A^T)\Pi\mathbb  E[X]+(C^{T}+\bar  C^{T})P\bigg[(C+\bar C)\mathbb E[X]
  + (D+\bar D)\mathbb E[u]\bigg]\\
&+\displaystyle\int_ Z(E^T(\theta)+ \bar E^T(\theta))P\bigg[(E(\theta)+\bar E(\theta))\mathbb E[X]+ (F(\theta)+\bar F(\theta))\mathbb E [u]\bigg]\nu
(d\theta)+2(Q+\bar Q)\mathbb E[X(t)]
\\=&\bigg [\dot{\Pi}+\Pi(t)(A+\bar A)
  +(A^T+\bar  A^T)\Pi+(C^{T}+\bar  C^{T})P(C+\bar C)
  +\displaystyle\int_ Z(E^T(\theta)+ \bar E^T(\theta))P(E(\theta)+\bar E(\theta))+2(Q+\bar Q)\bigg]  \mathbb  EX
  \\&+\bigg[\Pi(B+\bar B)+(C^{T}+\bar  C^{T})P (D+\bar D)+\displaystyle\int_ Z(E^T(\theta)+ \bar E^T(\theta))P (F(\theta)+\bar F(\theta))\nu
(d\theta)\bigg]\mathbb  E[u].
  \end{split}
\end{eqnarray}
Therefore,  by  putting 
 \eqref{eq:4.21} into \eqref{eq:4.23}, we conclude that  $\Pi(\cdot)$  should be the solutions to the
following Riccati equation

\begin {equation}\label{eq:4.24}
\left\{\begin{array}{lll}
 &\dot{\Pi}+\Pi(A+\bar A)
  +(A^\top+\bar  A^\top)\Pi+(C^{\top}+\bar  C^{\top})P(C+\bar C)
  +\displaystyle\int_ Z(E^\top(\theta)+ \bar E^\top(\theta))P(E(\theta)\nu
(d\theta)+\bar E(\theta))\nu
(d\theta)
\\&\quad+2(Q+\bar Q)-\bigg[\Pi(B+\bar B)+(C^{\top}+\bar  C^{\top})P (D+\bar D)+\displaystyle\int_ Z(E^\top(\theta)+ \bar E^\top(\theta))P (F(\theta)+\bar F(\theta))\nu
(d\theta)\bigg]\Sigma_2^{-1}
\\&\quad\cdot \bigg[(B^\top
    +\bar B^\top) \Pi+ (D^{T}
 +\bar D^{\top}) P(C+\bar C)+\int_ Z(F^\top(\theta)+\bar F^\top(\theta))P(E(\theta)+\bar E(\theta))\nu
(d\theta)\bigg]=0.
 \\&
 \Pi=G+\bar G
\end{array}
  \right.
  \end {equation}
By [28], under Assumptions
 \ref{ass:1.1} and \ref{ass:1.2}, we know that
Riccati equations \eqref{eq:4.18} and
\eqref{eq:4.24} has a unique solution,
respectively.

\subsection{State Feedback Representation}
In this section,  we establish  strictly 
the link  between the stochastic Hamilton system\eqref {eq:3.7}
and  Riccati  equation,  and show  the optimal 
control has a state feedback representation 
and the value function are expected to be given in terms
of the solution to the Riccati equation. Now
 we state our main result as follows.
\begin{thm}
  Let Assumption\ref{ass:1.1} and \ref{ass:1.2} be satisfied.
  Suppose that $( u(\cdot),  X(\cdot), p(\cdot), q(\cdot), r(\cdot, \cdot))$ is  the solution to the stochastic Hamilton system
  \eqref{eq:3.7}. Let
  $P(\cdot)$ and $\Pi(\cdot)$ be the solution
  to the   Riccati equations \eqref{eq:4.18} and
\eqref{eq:4.24}, respectively.
  Then 
  the optimal control $u(\cdot)$ has the following state feedback representation:
\begin{eqnarray}  \label{eq:4.25}
    \begin{split}
       u
 =&-\Sigma_0^{-1}\bigg(B^\top P(s)
+ D^{\top}PC+\int_ Z(F^\top(\theta)PE(\theta)\nu (d\theta)\bigg)(X-\mathbb E[X])
 \\&-\Sigma_2^{-1}\bigg[(B^\top
    +\bar B^\top) \Pi+ (D^{T}
 +\bar D^{\top}) P(C+\bar C)+\int_ Z(F^\top(\theta)\nu(d\theta)+\bar F^\top(\theta))P(E(\theta)\nu(d\theta)+\bar E(\theta))\bigg]\mathbb E[X]
    \end{split}
  \end{eqnarray}
  and the following  relations hold:
  \begin{eqnarray} \label{eq:4.27}
  \begin{split}
    p=P(X-\mathbb EX)+\Pi\mathbb  EX,
  \end{split}
\end{eqnarray}

  \begin{eqnarray}
  \begin{split}
    q=&\bigg[  PC
  -PD\Sigma_0^{-1}\bigg(B^\top P
+ D^{\top }PC+\int_ Z(F^\top (\theta)PE(\theta)\nu (d\theta)\bigg)\bigg](X-\mathbb EX)
\\&+\bigg[ P  (C+\bar C)
  -P(D+\bar D)\Sigma_2^{-1}\Big((B^\top
    +\bar B^\top) \Pi+ (D^{T}
 +\bar D^{\top}) P(C+\bar C)
 \\&~~~~~~+\int_ Z(F^\top(\theta)+\bar F^\top(\theta))P(E(\theta)+\bar E(\theta))\Big)\bigg]\mathbb E[X],
  \end{split}
\end{eqnarray}

\begin{eqnarray} \label{eq:4.28}
  \begin{split}
    r(\theta)=&\bigg[  P E(\theta)
  -PF(\theta)\Sigma_0^{-1}\bigg(B^\top P
+ D^{\top}PC+\int_ Z(F^\top(\theta)PE(\theta)\nu (d\theta)\bigg)\bigg](X-\mathbb EX)
\\&+\bigg[ P(E(\theta)+\bar E(\theta))
  -P(F(\theta)+\bar F(\theta))\Sigma_2^{-1}\Big((B^\top
    +\bar B^\top) \Pi+ (D^{\top}
 +\bar D^{\top}) P(C+\bar C)
 \\&~~~~~~+\int_ Z(F^\top(\theta)+\bar F^\top(\theta))P(E(\theta)+\bar E(\theta))\Big)\bigg]\mathbb E[X].
  \end{split}
\end{eqnarray}
Moreover, 
 \begin{eqnarray} \label{eq:4.29}
   \begin{split}
     \displaystyle \inf_{v(\cdot)\in \cal A}J(x,v(\cdot))=\frac{1}{2}\langle \Pi(0)x, x\rangle .
   \end{split}
 \end{eqnarray}

\end{thm}

\begin{proof}
Let $P(\cdot)$ and $\Pi(\cdot)$ be the unique solution
  to the  Riccati equations \eqref{eq:4.18} and
\eqref{eq:4.24}, respectively. 
Consider the following mean-field SDE
\begin{equation}\label{eq:4.30}
\left\{\begin {array}{ll}
  dX^*(t)=&(A(t)X^*(t)+\bar A(t)\mathbb E [X^*(t)]
  +B(t)u^*(t)+\bar B(t)\mathbb E [u^*(t)])dt
  \\&+(C(t)X^*(t)
  +\bar C(t)\mathbb E [X^*(t)]
  +D(t)u^*(t)+\bar D(t)\mathbb E [u^*(t)])dW(t)
  \\&+\displaystyle
  \int_{Z}(E(t, \theta)X^*(t-)+\bar E(t, \theta)\mathbb E [\hat X^*(t-)]
  +F(t,\theta)u^*(t)+\bar F(t, \theta)\mathbb E[u^*(t)])\tilde{\mu}(d\theta,
  dt),
   \\x(0)=&x \in \mathbb R^n,
\end {array}
\right.
\end{equation}
where 
\begin{eqnarray}
    \begin{split}
       u^*
 =&-\Sigma_0^{-1}\bigg(B^\top P(s)
+ D^{\top}PC+\int_ Z(F^\top(\theta)PE(\theta)\nu (d\theta)\bigg)(X^*-\mathbb E[X^*])
 \\&-\Sigma_2^{-1}\bigg[(B^\top
    +\bar B^\top) \Pi+ (D^{T}
 +\bar D^{\top}) P(C+\bar C)+\int_ Z(F^\top(\theta)\nu(d\theta)+\bar F^\top(\theta))P(E(\theta)\nu(d\theta)+\bar E(\theta))\bigg]\mathbb E[X^*].
    \end{split}
  \end{eqnarray}
  The solution to \eqref{eq:4.30} is 
 denoted by $X^*(\cdot).$
  Define
  \begin{eqnarray} \label{eq:4.32}
  \begin{split}
    p^*=P(X^*-\mathbb E[X^*])+\Pi\mathbb E[X^*],
  \end{split}
\end{eqnarray}

  \begin{eqnarray}
  \begin{split}
    q^*=&\bigg[  PC
  -PD\Sigma_0^{-1}\bigg(B^\top P
+ D^{\top }PC+\int_ Z(F^\top (\theta)PE(\theta)\nu (d\theta)\bigg)\bigg](X^*-\mathbb E[X^*])
\\&+\bigg[ P  (C+\bar C)
  -P(D+\bar D)\Sigma_2^{-1}\Big((B^\top
    +\bar B^\top) \Pi+ (D^{T}
 +\bar D^{\top}) P(C+\bar C)
 \\&~~~~~~+\int_ Z(F^\top(\theta)+\bar F^\top(\theta))P(E(\theta)+\bar E(\theta))\Big)\bigg]\mathbb E[X^*],
  \end{split}
\end{eqnarray}

\begin{eqnarray} \label{eq:4.34}
  \begin{split}
    r^*(\theta)=&\bigg[  P E(\theta)
  -PF(\theta)\Sigma_0^{-1}\bigg(B^\top P
+ D^{\top}PC+\int_ Z(F^\top(\theta)PE(\theta)\nu (d\theta)\bigg)\bigg](X^*-\mathbb E[X^*])
\\&+\bigg[ P(E(\theta)+\bar E(\theta))
  -P(F(\theta)+\bar F(\theta))\Sigma_2^{-1}\Big((B^\top
    +\bar B^\top) \Pi+ (D^{\top}
 +\bar D^{\top}) P(C+\bar C)
 \\&~~~~~~+\int_ Z(F^\top(\theta)+\bar F^\top(\theta))P(E(\theta)+\bar E(\theta))\Big)\bigg]\mathbb E[X^*].
  \end{split}
\end{eqnarray} 
  Then following the derivation of riccati equations \eqref{eq:4.18} and
\eqref{eq:4.24}
in the previous subsection, 
by applying It\^{o} formula to
$p^*=P(X^*-\mathbb E[X^*])+\Pi\mathbb E[X^*],$
we get that $(p^*(\cdot), q^*(\cdot), r^*(\cdot))$ satisfies the following 
backward stochastic differential equation

\begin {equation}\label{eq:4.35}
\left\{\begin{array}{lll}
dp^*(t)&=&-\bigg[A^\top(t)p^*(t)+\bar  A(t)^\top\mathbb E [p^*(t)]+C^{\top}(t)q^*(t)+\bar  C^{\top}(t)\mathbb E[q^*(t)]+\displaystyle\int_ ZE^\top(t,\theta)r^*(t,\theta)\nu (d\theta)\\&&+\displaystyle\int_ Z \bar E^\top(t,\theta)\mathbb E[r^*(t,\theta)]\nu
(d\theta)+2Q(t)X^*(t)+2\bar Q(t)\mathbb E[X^*(t)]\bigg]dt
+q^*(t)dW(t)+\displaystyle\int_ Z r^*(t,\theta)\tilde{\mu}(d\theta, dt),
 \\p^*(T)&=&2GX^*(T)+2\bar G\mathbb E[X^*(T)],
\end{array}
  \right.
\end {equation}
which is the adjoint 
equation associated 
with $(u^*(\cdot), X^*(\cdot)).$
Moreover, it is easy to check that
\begin{eqnarray} \label{eq:4.36}
  \begin{split}
 &2N(t)u^*(t)+2\bar N(t)\mathbb E[u^*(t)]+B^\top(t)p^*({t-})+\bar B^\top(t) \mathbb E [p^*({t-})]+ D^{\top}(t)q^*(t)
 +\bar D^{\top}(t) \mathbb E [q^*(t)]
\\&~~~~~~~~~~~+\int_ ZF^\top(t,\theta)r^*(t,\theta)\nu (d\theta)+\int_ Z
\bar F^\top(t,\theta) \mathbb E [r^*(t,\theta)]\nu (d\theta)=0, \quad a.e. a.s..
     \end{split}
\end{eqnarray}
Thus,  by Theorem \ref{thm:b2}, 
in terms of \eqref{eq:4.36}, we know that $ u^*(\cdot)$ is the corresponding optimal control and    $( u^*(\cdot),  X^*(\cdot),\\ p^*(\cdot), q^*(\cdot), r^*(\cdot, \cdot))$ is the solution to the stochastic Hamilton system \eqref{eq:3.7}. Therefore,  from  the 
uniqueness of the solution to the stochastic Hamilton system \eqref{eq:3.7}, we get that
$( u(\cdot),  X(\cdot), p(\cdot), q(\cdot), r(\cdot, \cdot))=( u^*(\cdot),  X^*(\cdot), p^*(\cdot), q^*(\cdot), r^*(\cdot, \cdot))$
which implies  that \eqref{eq:4.25}-\eqref{eq:4.28}
holds.

Now we  begin to   prove
\eqref{eq:4.29}.
In fact, since $u(\cdot)$ is 
the optimal control,
\begin{eqnarray} \label{eq:4.37}
   \begin{split}
     \displaystyle \inf_{v(\cdot)\in \cal A}J(x,v(\cdot))=&J(x,u(\cdot))
     \\=&\displaystyle \mathbb E\bigg[\int_0^T\bigg(\langle Q(t)X(t),
X(t)\rangle+ \langle \bar {
Q}(t)\mathbb E[X(t)], \mathbb E[X(t)]\rangle +\langle N(t)u(t), u(t)\rangle\\&+\langle \bar
{N}(t)\mathbb E[u(t)], \mathbb E[u(t)]\rangle \bigg)dt\bigg]+\mathbb E\langle
GX(T), X(T)\rangle +\langle \bar{ G}\mathbb E[X(T)], \mathbb
E[X(T)]\rangle.
   \end{split}
 \end{eqnarray}
 On the other hand, applying It\^{o}
 formula to $\langle p(t), X(t) \rangle$,
 we get that 
 
 \begin{eqnarray} \label{eq:4.38}
   \begin{split}
    &2\mathbb E\bigg[\int_0^T\bigg(\langle Q(t)X(t),
X(t)\rangle+ \langle \bar {
Q}(t)\mathbb E[X(t)], \mathbb E[X(t)]\rangle\bigg)dt\bigg] +2\mathbb E\bigg[\langle
GX(T), X(T)\rangle \bigg]+2\langle \bar{ G}\mathbb E[X(T)], \mathbb
E[X(T)]\rangle
\\=&\mathbb
E\bigg [\int_0^T
\bigg(\langle p(t), B(t)u(t)+\bar B(t)\mathbb E [u(t)]\rangle+
\langle q(t), D(t)u(t)+\bar D(t)\mathbb E[u(t)]\rangle
\\& \quad \quad +\int_{z}\langle r(t,\theta), F(t,\theta)u(t)+\bar F(t,\theta)\mathbb E[u(t)]\rangle \nu d(\theta)\bigg)dt\bigg ]+\mathbb E
 \langle p(0),x \rangle
\\&=\mathbb
E\bigg[\int_0^T \bigg \langle B^\top (t)p({t-})+\bar B^\top(t) \mathbb E [p({t-})]+ D^{\top}(t)q(t)
 +\bar D^{\top}(t) \mathbb E [q(t)] 
\\&\quad\quad\quad\quad\quad+\int_ ZF^\top(t,\theta)r(t,\theta)\nu (d\theta)+\int_ Z
\bar F^\top(t,\theta) \mathbb E [r(t,\theta)]\nu (d\theta), u(t)\bigg\rangle dt\bigg]+ \mathbb E
 \langle p(0),x \rangle.
   \end{split}
 \end{eqnarray}
 Then putting \eqref{eq:4.38}
into \eqref{eq:4.37}
and in terms of \eqref{eq:3.1000}, we get that
 \begin{eqnarray} \label{eq:4.39}
   \begin{split}
     \displaystyle \inf_{v(\cdot)\in \cal A}J(x,v(\cdot))=& 
     \frac{1}{2}\langle p(0),x \rangle+
     \frac{1}{2}\mathbb
E\bigg[\int_0^T \bigg \langle B^\top (t)p({t-})+\bar B^\top(t) \mathbb E [p({t-})]+ D^{\top}(t)q(t)
 +\bar D^{\top}(t) \mathbb E [q(t)]
\\&+\int_ ZF^\top(t,\theta)r(t,\theta)\nu (d\theta)+\int_ Z
\bar F^\top(t,\theta) \mathbb E [r(t,\theta)]\nu (d\theta)+2N(t)u(t)+2\bar N(t)\mathbb E[u(t)], u(t)\bigg\rangle dt\bigg]
\\=&\frac{1}{2}\langle p(0),x \rangle
   \end{split}
 \end{eqnarray}
 From \eqref{eq:4.27},
 we have
\begin{eqnarray}\label{eq:4.40}
  p(0)=\Pi(0) x
\end{eqnarray}
Therefore,
putting \eqref{eq:4.40} into
\eqref{eq:4.39} leads to

\begin{eqnarray} \label{eq:4.41}
   \begin{split}
     \displaystyle \inf_{u(\cdot)\in \cal A}J(x,v(\cdot)) =\frac{1}{2} 
     \langle \Pi(0) x,x \rangle.
   \end{split}
 \end{eqnarray}
The proof is complete.
\end{proof}

\end{document}